\documentclass{amsart}
\usepackage{amssymb}
\usepackage{amsmath,amsfonts,amsthm,amstext}
\usepackage{stmaryrd}
\usepackage[all]{xy}
\newtheorem{theorem}{Theorem}[section]
\newtheorem{lemma}[theorem]{Lemma}

\newtheorem{cor}[theorem]{Corollary}

\theoremstyle{definition}

\def\mapnew#1{\smash{\mathop{\longrightarrow}\limits^{#1}}}

\def\mapnew#1{\smash{\mathop{\longrightarrow}\limits^{#1}}}

\usepackage{color}

\definecolor{red}{rgb}{1,0,0}

\begin{document}

\title[]{The Fibonacci Lie algebra is not finitely presented }

\author{Dessislava H. Kochloukova}
\address
{Department of Mathematics, State University of Campinas (UNICAMP), 13083-859, Campinas, SP, Brazil}
\email{desi@ime.unicamp.br}

\author{Victor Petrogradsky}
\address{Department of Mathematics, University of Brasilia, 70910-900 Brasilia DF, Brazil}
\curraddr{}
\email{petrogradsky@rambler.ru}
\thanks{The first author was partially supported by
 CNPq 305457/2021-7 and  FAPESP
18/23690-6.
The second author was partially supported by grant FAPESP 2022/10889-4. }


\dedicatory{}

\commby{}

\begin{abstract} We prove that the  Fibonacci Lie algebra
 and the related just infinite self-similar Lie algebra
are not finitely presented.
\end{abstract}

\maketitle


\bibliographystyle{amsplain}

\section{Introduction}

In this paper we study so called {\it Fibonacci Lie algebra} $L$ introduced by Petrogradsky in~\cite{Petr},
which is a natural analogue of the Grigorchuk group~\cite{Grigorchuk80}.
Further properties of this Lie algebra and the related restricted Lie algebra were studied in~\cite{Petr-Shest, Petr-Shest2}.
More fractal Lie algebras and superalgebras were constructed in~\cite{Bartholdi,PeOtto,Pe16,Shest-Zel,P-S-Z}.

There is a vast literature on self-similar groups.
A countable set of relations of the Grigorchuk group was found by Lysionok~\cite{Lys85},
it is obtained from a finite set by applying iterations of an endomorphism.
Later, such presentations were called finite {\it endomorphic presentations} or $L$-{\it presentations} by Bartholdi~\cite{Bartholdi03}.
The notion of self-similarity in case of Lie algebras was suggested and developed by Bartholdi~\cite{Bartholdi}.
Further relations of self-similarity with so called virtual endomorphisms were studied in~\cite{F-K-S}.

In this paper we show using homological methods that the Fibonacci Lie algebra is not finitely presented.
The conjecture that the {\it restricted} Fibonacci Lie algebra is infinitely presented was formulated in~\cite{Petr-Shest2}
along with an infinite set of relations obtained by iterations of an endomorphism  applied to a finite set.

It is well known that finite dimensional simple Lie algebras
over an algebraically closed field of characteristic zero are finitely presented, namely, they are determined by Serre's relations.
More generally,  the Kac-Moody algebras are defined by similar finitely many relations in terms of the Cartan matrix~\cite{Kac}.
The simple infinite dimensional Witt Lie algebra $W_1$ in case of characteristic zero is finitely presented~\cite{Stewart75}.
The fact that the positive part of $W_1$ is finitely presented was remarked in~\cite{Ufn80}.
More generally, simple infinite dimensional Cartan type Lie algebras $W_n,S_n,K_n,H_n$ in characteristic zero are finitely presented~\cite{LeiPol}.
Presentations of Lie algebras and close objects are extensively studied in terms of Gr\"obner-Shirshov bases,
in particular this approach was applied to simple finite dimensional Lie algebras and Kac-Moody algebras,
see \cite{BokChen07,BokChen14,BokKle96}.
Finitely presented Lie algebras and intermediate growth of their universal enveloping algebras were studied in~\cite{Kos15, Kos17}.

In the case of $\mathbb{N}$-graded Lie algebras finite presentability is equivalent with both the first and the second homology being finite dimensional.
In a recent paper using homological methods  Kochloukova and Martinez-Perez proved that if $L$ is a finitely presented
$\mathbb{N}$-graded Lie algebra without non-abelian free Lie subalgebras
then the derived Lie subalgebra $[L,L]$ is finitely generated.
Unfortunately this cannot be used to prove that  the Fibonacci Lie algebra is not finitely presented since its derived subalgebra is finitely generated,
see Lemma \ref{derived-fin-gen}.

At this stage we define the Fibonacci Lie algebra as $L = \langle v_1, v_2 \rangle$, where $v_1, v_2 \in Der(R)$
for $R = k[t_0, t_1, t_2, \ldots]/ (t_0^2, t_1^2, t_2^2, \ldots )$ and $k$ a field of characteristic $p = 2$ and
$v_i =  \partial_i + t_{i-1} ( \partial_{i+1} + t_i(\partial_{i+2} +  t_{i+1}(\partial_{i+3} + \ldots)))$,
where $\partial_i = \partial/ \partial t_i$.
We consider also a related Lie algebra $\widetilde L=\langle \partial_1,v_2\rangle\subset Der R$,
which is just-infinite~\cite{Petr-Shest2}.
This algebra was defined as a self-similar Lie algebra by Bartholdi~\cite{Bartholdi}.
The Lie algebras $L$ and $\widetilde{L}$ are $\mathbb{N}$-graded,
we describe their properties in more details in Section \ref{def-petro}.

\medskip
\noindent
{\bf Main Theorem.} {\it   Let $p = 2$. The Fibonacci Lie algebra $L$ and the related self-similar
 just infinite Lie algebra $\widetilde{L}$ are  not finitely presented. }

\medskip

Recently Futorny, Kochloukova and Sidki  constructed in \cite[Thm.~D]{F-K-S} examples of finitely presented,
infinite dimensional, metabelian,  self-similar Lie algebras.
In the case of groups Kochloukova and Sidki constructed in \cite{K-S} finitely presented  $S$-arithmetic self-similar groups of homological type $FP_n$.
In contrast the Schur multiplier of the Grigorchuk group is infinitely generated,
 in particular this explains why the Grigorchuk group is a self-similar group that is not finitely presented \cite{Grig99}.
Another famous self-similar group, the Gupta-Sidki group defined in \cite{G-S}, is also infinitely presented~\cite{Sidki87}.

\section{ Preliminaries on homological properties of Lie algebras}

\subsection{Basic homological algebra over associative ring $R$ with unity} \label{homology-theory}
We review some basic definition from homological algebra that work for an arbitrary associative ring $R$ with unity.
For more results the reader is referred to \cite{Rotman} and \cite{Weibel}.
A complex of $R$-modules
$${\mathcal M} : \ldots \longrightarrow  M_{i+1} \mapnew{ d_{i+1}} M_i \mapnew{d_i} M_{i-1} \mapnew{ d_{i-1}} \ldots
$$
is a sequence of homomorphisms such that $d_{i-1} d_i = 0$. The $i$-th homology group of ${\mathcal M}$ is
$H_i({\mathcal M}) = Ker (d_i) / Im (d_{i+1}).$
The complex is called exact if all homology groups are trivial.

Let $V$ be a right $R$-module, $W$ be a left $R$-module and
$${\mathcal P} : \ldots \to P_i \to P_{i-1} \to \ldots \to P_0 \to V \to 0
$$ be a projective resolution of right $R$-modules  (i.e. the sequence is exact and $P_i$ is projective for all $i \geq 0$)
and
$${\mathcal S} : \ldots \to S_i \to S_{i-1} \to \ldots \to S_0 \to W \to 0
$$ be a projective resolution of left $R$-modules (i.e. the sequence is exact and $S_i$ is projective for all $i \geq 0$).
Note that free modules are projectives, so it is natural to consider free resolutions too.
The deleted projective resolution ${\mathcal P}_V$ (resp. ${\mathcal S}_W$) is the complex obtained from ${\mathcal P}$ (resp. $\mathcal{S}$)
substituting $V$ ( resp. $W$) with the trivial module $0$. Then ${\mathcal P}_V \otimes_R  W$ and $V \otimes_R {\mathcal S}_W$ are both  complexes.
One of the basic facts in homological algebra (see \cite{Rotman}) is
that $H_i({\mathcal P}_V \otimes_R  W) \simeq H_i(V \otimes_R {\mathcal S}_W)$
and by definition
$Tor^R_i(V, W) = H_i({\mathcal P}_V \otimes_R  W) \simeq H_i(V \otimes_R {\mathcal S}_W)$
is an abelian group i.e. $\mathbb{Z}$-module. Note that
$Tor^R_0(V, W)  = V \otimes_R W.$

\subsection{Basic homological algebra for Lie algebras} \label{Lie-homology}
Let $L$ be an arbitrary Lie algebra over a field $k$ and $W$ be a left $U(L)$-module.
Then by definition the $i$-th homology of $L$ with coefficients in $W$ is
$$ H_i(L, W) : = Tor_i^{U(L)}(k, W),$$
where $k$ is the basic field considered as a trivial right $U(L)$-module (i.e.  $L$ acts as 0).
If $W$ is a right $U(L)$-module
$$ H_i(L, W) = Tor_i^{U(L)}(W,k),$$
where $k$ is the  basic field considered as a trivial left $U(L)$-module.

Consider the  Chevalley-Eilenberg  free resolution of the trivial right $U(L)$-module $k$
$$
{\mathcal L} : \ldots \to \wedge^{n} L \otimes  U(L) \mapnew{\partial_n} \wedge^{n-1} L \otimes U(L) \to \ldots \to L \otimes U(L) \to U(L) \to k \to 0
$$
with differential
$$
\partial_n( x_1 \wedge \ldots \wedge x_n \otimes h) = $$
$$\sum_{1 \leq j \leq n}\!\!\! (-1)^{j+1} x_1 \wedge \ldots \widehat{x_j} \ldots \wedge x_n \otimes x_j h
+\!\!\!\!\sum_{1 \leq i < j \leq n}\!\!\! (-1)^{i+j} [x_i, x_j] \wedge x_1 \wedge \ldots \widehat{x_i} \ldots \widehat{x_j} \ldots \wedge x_n \otimes h$$
where $x_1, \ldots, x_n \in L, h \in U(L)$. Recall that ${\mathcal L}_k$ is the complex obtained from ${\mathcal L}$
substituting the trivial $U(L)$-module $k$ with the zero module.
Thus
$$
{\mathcal L_k} {\otimes_{ U(L)}} k : \ldots \to \wedge^n L \mapnew{d_n} \wedge^{n-1} L \mapnew{d_{n-1}} \ldots \to L \to  k \to 0
$$
has differential
$$
d_n( x_1 \wedge \ldots \wedge x_n ) =
\sum_{1 \leq i < j \leq n}(-1)^{i+j} [x_i, x_j] \wedge x_1 \wedge \ldots \widehat{x_i} \ldots \widehat{x_j} \ldots \wedge x_n $$
The $n$th-homology of $L$ is
$
H_n(L,k) = Ker (d_n) / Im (d_{n+1}).
$
The definition of homology implies for an arbitrary left $U(L)$-module $V$ that
$
H_0(L, V) \simeq k \otimes_{U(L)} V = V / L V = : V_L,
$
called {\it coinvariants.}
We write
$ H^0(L,V)$ for  $\{ v \in V \mid lv = 0 \hbox{ for } l \in L \} = : V^L,$
called {\it invariants.}
The same works for an arbitrary right $U(L)$-module $V$ i.e.
$
H_0(L, V) = V \otimes_{U(L)} k = V / V L = : V_L.
$
We write
$
 H^0(L,V)$ for $ \{ v \in V \mid vl = 0 \hbox{ for } l \in L \} = : V^L.$

\begin{lemma} \label{duality} Let $Q$ be a one dimensional Lie algebra and $V$ be a   $U(Q)$-module.
   Then there is a natural isomorphism
	$\tau_V : H_1(Q, V) \to  H^0(Q, V)$. The naturality means that if $f : V \to W$ is a homomorphism
of $U(Q)$-modules then for the induced maps $f_1 : H_1(Q, V) \to H_1(Q, W)$ and $f_2: H^0 (Q, V) \to H^0(Q, W)$
we have that $\tau_W \circ f_1 = f_2 \circ \tau_V$.
	\end{lemma}
\begin{proof}
	Let $\mu : V \to V$ be the map that sends an element to its image under that action of a fixed element $q \in Q \setminus \{ 0 \}$.
	Consider the free resolution of the trivial $U(Q)$-module $k$
	$${\mathcal{P}} : 0 \to U(Q) \mapnew{\alpha} U(Q) \mapnew{\epsilon} k \to 0
	$$
	where the map $\alpha$ is multiplication by $q$ and the map $\epsilon$ is the augmentation map. Denote by  ${\mathcal{P}}_k$
the deleted complex obtained from $\mathcal{P}$ by substituting the module $k$ in dimension -1 with the zero module. Then we get a complex
	$${\mathcal P}_k\otimes_{U(Q)} V : 0 \to V \mapnew{\mu} V \to 0$$
	and $H_1(Q, V) = H_1( {\mathcal P}_k \otimes_{U(Q)} V ) \simeq Ker (\mu) = V^Q =  H^0(Q, V).$
\end{proof}

\subsection{$\mathbb{N}$-graded Lie algebras} 

We write $\mathbb{N}$ for the set of positive integers. A Lie algebra $L$ is $\mathbb{N}$-graded if it has a decomposition
$$
L = L_1\oplus L_2 \oplus \ldots \oplus L_n \oplus \ldots,\qquad \hbox{ where } [L_i, L_j] \subseteq L_{i+j},\ i,j\ge 1.$$
A $\mathbb{N}$-graded $L$-module $V$ has a decomposition
$$V = V_1\oplus V_2 \oplus \ldots \oplus V_n \oplus \ldots,\qquad \hbox{ where } V_i L_j \subseteq V_{i + j},\ i,j\ge 1. $$

The following result is known, but as we could not find a reference, we include a proof.

\begin{lemma} \label{Ngraded}
Let $L$ be a $\mathbb{N}$-graded Lie algebra. Then

a) $L$ is finitely generated if and only if $H_1(L,k)$ is finite dimensional.

b) $L$ is finitely presented if and only if both $H_1(L,k)$ and $H_2(L,k)$ are  finite dimensional.
\end{lemma}

\begin{proof}

a) Since $L$ is  $\mathbb{N}$-graded, $L$ is finitely generated if and only if $L/ [L,L]$ is finite dimensional but $H_1(L,k) \simeq L/ [L,L]$.

b) Suppose $L$ is finitely generated and $H_2(L,k)$ is  finite dimensional.
 Recall that a Lie algebra $L$ is of  homological type $FP_n$ if the trivial  $U(L)$-module $k$  has a projective resolution with
all projectives in dimension $\leq n$ finitely generated.
By \cite[Cor. 2.3]{K-MP}  $L$ is of homological type $FP_n$ if and only if $H_i(L,k)$ is finite dimensional for $i \leq n$.
The case $n = 2$ was earlier proved in  \cite[section~2]{Thomas}.
Then $L$ is of type $FP_2$.

If $L \simeq F/ R$  is an isomorphism of $\mathbb{N}$-graded Lie algebras, where $F$ is a finitely generated,
free Lie algebra  that is equipped  with a $\mathbb{N}$-grading and $R$ is a $\mathbb{N}$-graded ideal of  $F$.
Then by \cite[Lemma 2.4]{K-MP2} $L$ is of type $FP_2$ if and only if $R/ [R,R]$ is finitely generated as $U(L)$-module.
Note that the $U(L)$-action on $R/ [R,R]$ is induced by the adjoint action of $F$ on $R$.
Finally since $R$ is a $\mathbb{N}$-graded ideal of $F$ we have that $R$ is finitely generated as an $U(F)$-module (via the adjoint action)
if and only if $R/ [R, F]$ is finite dimensional.
But the latter follows from the fact that $R/ [R,R]$ is finitely generated as $U(L)$-module and the former is equivalent to $L$ is finitely presented.
\end{proof}


\section{Properties of the Fibonacci Lie algebra $L$ and $\widetilde L$} \label{def-petro}

We recall basic properties of {\it Fibonacci algebra}
introduced in~\cite{Petr}, subsequently studied in more details  in~\cite{Petr-Shest,Petr-Shest2}.
By $k$ we denote the base field of characteristic 2.
Let $R:= k[t_i| i\ge 0 ]/ (t_i^2| i\ge 0 )$, the ring of
the truncated polynomials.
Denote derivations $\partial_i:=  \partial / \partial t_i\in Der (R)$, $i\ge 0$.
Define recursively $\{ v_i | i \geq 1 \} \subset Der (R)$ so called {\it pivot elements}:
$$ v_i = \partial_i + t_{i-1} v_{i+1}, \qquad i \geq 1. $$
Equivalently, one can consider explicit presentations as infinite sums
\begin{equation} \label{eqeq}
v_i =  \partial_i + t_{i-1} ( \partial_{i+1} + t_i(\partial_{i+2} + \ldots)),\qquad i\ge 1.
\end{equation}

Let $S$ be a subset in a Lie algebra. Then $\langle S\rangle $ denotes the Lie subalgebra generated by $S$.
Now let $L: = \langle v_1, v_2 \rangle$, the Lie subalgebra of $Der(R)$ generated by $\{v_1,v_2\}$
and referred to as the {\it Fibonacci Lie algebra}.
Commutators of the of pivot elements are~\cite{Petr}:
\begin{equation}\label{prod_pivot}
\begin{split}
&[v_i, v_{i+1}] = v_{i+2},\qquad\qquad i\ge 1;\\
&[v_i, v_j] = t_{i-1} t_i \cdots t_{j-3} v_{j+1}, \qquad 1\le i<j.
\end{split}
\end{equation}
The explicit basis of $L$ was determined in~\cite{Petr-Shest}
\begin{equation} \label{basisL}
\{ v_1, v_2, v_3,  t_0^{z_0} \ldots  t_i^{z_i} v_{i+4} \mid i \geq 0,\  z_0, \ldots z_i   \in \{ 0,1 \} \},
\end{equation}
Call by a {\it tail} a product  $$r_n:=t_0^{z_0}\cdots t_n^{z_n}$$ for $n\ge 0$.
In case $n<0$ consider that $r_n=1$.
So, basis monomials~\eqref{basisL} are shortly written as $\{ r_{n-4}v_n|n\ge 1 \}$.

A commutator of two arbitrary basis monomials~\eqref{basisL} is computed as follows, see~\cite{Petr-Shest}.
Let $1\le n \leq  m$, then using~\eqref{prod_pivot} we have
\begin{align} \label{com123}
[r_{n-4}v_n, r_{m-4}v_m] &=r_{n-4}v_n(r_{m-4})v_m+ r_{n-4}r_{m-4}[v_n, v_m].
\end{align}
Indeed, $v_m(r_{n-4})=0$ because $v_m$ contains $\partial_i$, with $i\ge m>m-4  \geq  n-4$ and $r_{n-4}$ has only $t_j$ with $j\le n-4$.
Note that $v_n(t_j) = t_{n-1} \ldots t_{j-2}$ for $j \geq n+1$, $v_n(t_n) = 1$ and $v_n(t_j) = 0$ for $j < n$.

A restricted Lie algebra generated by the same elements $v_1,v_2$ has a nil $p$-mapping~\cite{Petr},
thus providing a natural analogue of the Grigorchuk group in the class of restricted Lie algebras in characteristic 2.
The degrees of nilpotence of all elements are not bounded.
Nevertheless, $v_i^4=0$, $i\ge 1$~\cite{Petr}, hence
 $ad(v_i)^4=0$ for all $i\ge 1$.

Consider those basis monomials of $L$~\eqref{basisL} that contain factor $t_0$, namely:
\begin{equation} \label{basisA}
X_A := \{ t_0  t_1^{z_1} \ldots t_i^{z_i}  v_{i+4} \mid i \geq 0,  z_1, \ldots z_i \in \{ 0, 1 \}  \}.
\end{equation}
\begin{lemma} \label{Adef}
Let $A$ be a linear span of $X_A$.
Then $A$ is an abelian ideal of $L$.
\end{lemma}
\begin{proof}   It follows directly from (\ref{com123}).
\end{proof}

Consider the mapping $\tau(t_i):=t_{i+1}$, $i\ge 0$, which naturally extends to a monomorphism $\tau:R\to R$.
Set also $\tau(\partial_i):=\partial_{i+1}$, $i\ge 1$.
We get an extension to a homomorphism on associative (or Lie) algebras generated by $\{t_i,\partial_i|i\ge 0\}$.
This extensions also applies to our infinite sums.
We refer to $\tau$ as the {\it shift mapping}.

\begin{lemma}  \label{iso}
The shift mapping $\tau:L\to L$ is a monomorphsim.
Denote $L_1:=\tau(L)\cong L$.
Also, the subalgebra $L_1$ is generated by $\{v_2, v_3\}$.
\end{lemma}
\begin{proof}
Observe that $\tau(v_i)=v_{i+1}$, $i\ge 0$.
The action of $\tau$ on basis monomials~\eqref{basisL} yields the shift of all indices by one.
So, the subalgebra $L_1=\tau(L)$ has a basis
\begin{equation} \label{basisL0}
X_0 := \{ v_2, v_3, v_4 ,   t_1^{z_1} \ldots t_{i+1}^{z_{i+1}}  v_{i+5}  \mid   i \geq  0, z_1, \ldots, z_i \in \{ 0 , 1 \}  \}.\qedhere
\end{equation}
\end{proof}

Let $Q_1:=kv_1$ be the one dimensional Lie subalgebra.
For a short exact sequence of Lie algebras $0 \to M' \to M \to M'' \to 0$ that splits
we use notation of a {\it semidirect product} $M = M' \leftthreetimes M''$.

\begin{lemma} \label{decomposition}
We get a decomposition
\begin{equation}\label{decomp0}
L = (A \leftthreetimes L_1) \leftthreetimes Q_1.
\end{equation}
\end{lemma}
\begin{proof}
Note  that  $X_A$, $X_0$, and $\{ v_1 \}$ are $k$-bases of $A$, $L_1$, and $Q_1$
respectively and the disjoint union $X_A \cup X_0 \cup \{ v_1 \}$ is a basis of $L$
(see \eqref{basisA}, \eqref{basisL0}, and \eqref{basisL}).
Then the Lie subalgebra generated by $A$ and $L_1$ is $A \leftthreetimes L_1$ and since
$[L_1, v_1] \subset A \leftthreetimes L_1$  we have that $A \leftthreetimes L_1$ is an ideal of $L$. 
\end{proof}

In~\cite{Bartholdi} Bartholdi called a Lie algebra $B$ {\it self-similar} if there exists a  monomorphism
\begin{equation} \label{psi}
\psi : B \to (X \otimes B) \leftthreetimes Der(X),
\end{equation}
where $X$ is a commutative $k$-algebra.
An algebra is called {\it just infinite} if any non-zero ideal is of finite codimension.
This notion serves as a weak version of simplicity for infinite dimensional algebras.

Now we define and describe properties of the Lie algebra considered in~\cite{Bartholdi,Petr-Shest2}.
\begin{theorem} \label{just_infinite}
Let $L$ be the Fibonacci Lie algebra.
Set $\widetilde L:=L/A$, where $A$ the ideal defined above. Then
\begin{enumerate}
\item  $\tilde L$ is embedded in $Der R $ and has a basis $\partial_1\cup X_0$ (\eqref{basisL0})~\cite[Lemma 7.2]{Petr-Shest2}.
\item $\tilde L$ is 2-generated by $\{\partial_1, v_2=\partial_2+t_1(\partial_3+t_2(\partial_4+\cdots )) \}$.
\item  $\widetilde L=k\partial_1\rightthreetimes L_1$, where $L_1=\langle v_2,v_3\rangle =\tau(L)\cong L$.
\item $\widetilde L$ is just infinite~\cite[Theorem 7.8]{Petr-Shest2}.
\item $\widetilde L$ is self-similar~\cite{Bartholdi}.
\end{enumerate}
\end{theorem}
\begin{proof}
1) Consider $J:=t_0 R\triangleleft R$, then $R/J\cong  \widetilde R=:k[t_i| i\ge 1]/(t_i^2|i\ge 1)\subset R$.
The ideal $J$ is invariant under $v_1,v_2$ because they cannot eliminate $t_0$.
Define the composed homomorphism $\pi:L\to  Der \widetilde R\hookrightarrow Der R$,
where the latter embedding assumes the trivial action on $t_0$ and keeps the same action on $t_i$, $i\ge 1$.
We get $\pi (v_1)=\pi(\partial_1+t_0 v_2)=\partial_1$, $\pi(v_2)=v_2$, and $\pi(X_A)=0$.
Also, the elements of $X_0$ remain the same and linearly independent.
Hence $Ker(\pi)=A$, and $Im(\pi)\cong L/A $ has a basis $\partial_1\cup X_0$.

2,3) $[\partial_1,v_2]=[\partial_1,\partial_2+t_1v_3]=v_3$, while $v_2,v_3$ generate $L_1$ by Lemma~\ref{iso}.
We get $\widetilde L=k\partial_1\rightthreetimes L_1\subset Der R$, where $L_1=\langle v_2,v_3\rangle =\tau(L)\cong L$.

5) Put $X: = k[t]/ (t^2)$, with respective $\partial$, and define a homomorphism
\begin{align*}
&\psi : \widetilde L   \to (X \otimes \tau(\widetilde L)) \leftthreetimes Der(X)\qquad \text{defined as}\\
&\psi(\partial_1) :=\partial,\quad
\psi(v_2)=\psi(\partial_2+t_1v_3) :=1\otimes \partial_2 +t\otimes v_3.
\end{align*}
It remains to identify $\langle \partial_2,v_3\rangle =\tau (\langle \partial_1,v_2 \rangle)=\tau(\widetilde L)$ with $\widetilde L$.
\end{proof}
\begin{theorem}[\cite{Petr-Shest2}]\label{Ngrading}
The algebras $L$, $\widetilde L$ are $\mathbb{N}$-graded by degree in the generators.
\end{theorem}
\begin{proof} The algebra $L$ is graded by degree in the generators and its
monomial basis~\eqref{basisL} is homogeneous~\cite{Petr-Shest,Petr-Shest2}.
Hence, $\widetilde L=L/A$ inherits the grading.
\end{proof}

\section{Homological argument}

Let $M = B \leftthreetimes C$ be a semidirect product of Lie algebras.
Consider the complex
$$
{\mathcal M} : \ldots \to \wedge^n M \mapnew{d_n} \wedge^{n-1} M \to \ldots \to M \to k \to 0
$$
with differential
$$
d_n( x_1 \wedge \ldots \wedge x_n ) =
 \sum_{1 \leq i < j \leq n}(-1)^{i+j} [x_i, x_j] \wedge x_1 \wedge \ldots \widehat{x_i} \ldots \widehat{x_j} \ldots \wedge x_n $$
 where $x_1, \ldots, x_n \in M$.
Thus
$ d_2(x_1 \wedge x_2) = - [x_1, x_2] $ and
\begin{equation}\label{diff3}
d_3(x_1 \wedge x_2 \wedge x_3) = - [x_1, x_2] \wedge x_3 + [x_1, x_3] \wedge x_2 - [x_2, x_3] \wedge x_1.
\end{equation}
By subsection \ref{Lie-homology} $H_i(M,k) \simeq H_i(\mathcal{M})$,  so
$H_2(M,k) \simeq H_2(\mathcal{M}).$
On the other hand
\begin{align*}
\wedge^2 M &= (\wedge^2 B) \oplus (B \otimes C) \oplus (\wedge^2 C );\\
\wedge^3 M &= (\wedge^3 B) \oplus ((\wedge^2 B) \otimes C) \oplus (B \otimes ( \wedge^ 2 C)) \oplus (\wedge^3 C).
\end{align*}

Consider the subcomplex $\mathcal{B}$ of $\mathcal{M}$
$$
\mathcal{B} :
\ldots \to \wedge^i B \to \wedge^{i-1} B \to \ldots \to B \to k \to 0
$$
 and set $\mathcal D$ to be the quotient complex $\mathcal{M} / \mathcal{B}$. Then we have the short exact sequence of
complexes
\begin{equation} \label{complexes-ses} 0 \to \mathcal{B} \to \mathcal{M} \to \mathcal{D} \to 0\end{equation}
and  by \cite[Thm. 6.10]{Rotman} it gives rise to a long exact sequence in homology
\begin{equation} \label{les1}
 \ldots \to H_3(\mathcal{D}) \mapnew{\partial_3}  H_2(\mathcal{B}) \to  H_2(\mathcal{M}) \to  H_2(\mathcal{D}) \mapnew{\partial_2}  H_1(\mathcal{B}) \to  H_1(\mathcal{M}) \to  H_1(\mathcal{D}) \to \ldots
 \end{equation}
 where $\partial_3$ and $\partial_2$ are the
connecting homomorphisms as defined in \cite[p. 333]{Rotman}.

Consider the subcomplex $\mathcal{E}$ of $\mathcal{M}$
$$
\mathcal{E} :
\ldots \to  E_i = \oplus _{1 \leq j \leq i}(\wedge^j B) \otimes (\wedge^{i-j} C) \to E_{i-1}  \to \ldots \to E_1 = B \to E_0= k \to 0
$$
and the subcomplex
 $\mathcal{C}$ of $\mathcal{M}$
 $$
\mathcal{C} :
\ldots \to \wedge^i C \to \wedge^{i-1} C \to \ldots \to C \to k \to 0
$$
Substituting in $\mathcal{C}$ the module $k$ with $0$ we get that complex $\widetilde{\mathcal{C}}$
and a short exact sequence of complexes
$0 \to \mathcal{E} \to \mathcal{M} \ \mapnew{ \pi}  \  \widetilde{\mathcal{C}} \to 0.$ 
Note that the above short exact sequence splits
\begin{equation} \label{soma} \mathcal{M} \simeq \mathcal{E}\oplus \widetilde{\mathcal{C}}. \end{equation}

\begin{lemma} \label{newMay}
a) The map $\pi$ induces an epimorphism
$
H_2(\mathcal{M} )\   \mapnew{ \pi_*} \ H_2( \widetilde{\mathcal{C}})
$
that can be identified with an epimorphism
$
H_2(M, k) \to H_2(C, k)
$
 induced by the projection $M \to C$.

b) The inclusion map $C \to M$ induces a monomorphsim $H_2(C,k)\to H_2(M,k)$.
\end{lemma}

\begin{proof} By (\ref{soma})
$H_2(\mathcal{M})=H_2(\mathcal{E})\oplus H_2(\widetilde{\mathcal{C}})$. Note that $H_2(\mathcal{M}) \simeq H_2(M,k)$ and $H_2(\widetilde{\mathcal{C}}) \simeq H_2(C,k)$.
\end{proof}

\begin{lemma}\label{split} Let $M = B \leftthreetimes C$ be a split extension of Lie algebras, where
 $C$ is one dimensional.  Then

 a) there is a short exact sequence
	$$0 \to H_0(C, H_2(B,k)) \to H_2(M,k) \to H_1(C, H_1(B,k)) \to 0$$
	
b)	
Identifying $H_2(M,k)$ with a subquotient of $M \wedge M$, if $b_ 0 \in B \setminus [B,B]$ and $c_0 \in C \setminus \{ 0 \}$ such that $[b_0, c_0] = 0$
then  $b_0 \otimes c_0 \in B \otimes C \subset M \wedge M $ represents a cycle whose homology class
in $H_2(M,k)$ maps to a non-zero element in $H_1(C, H_1(B,k))$.
\end{lemma}

{\it Remark}. For the readers that are accustomed with spectral sequence part a) of the above lemma is a special case of the Lie algebra version of the
Lyndon-Hoschild-Serre spectral sequence i.e. $E_{i,j}^2 = H_i(C, H_j(B,k))$ converging to $H_{i+j}(M, k)$. Here $E_{i,j}^2 = 0$ for $i \geq 2$,
hence the spectral sequence is concentrated in two consecutive columns, hence collapses i.e. $E_{i,j}^{\infty} = E_{i,j}^2$.

\begin{proof}
a)  Note that $\wedge^2 M = (\wedge^2 B) \oplus (B \otimes C) \oplus (\wedge^2 C ) =   (\wedge^2 B) \oplus (B \otimes C) $ and
$$\wedge^3 M = (\wedge^3 B) \oplus ((\wedge^2 B) \otimes C) \oplus (B \otimes ( \wedge^ 2 C)) \oplus (\wedge^3 C) =  (\wedge^3 B) \oplus ((\wedge^2 B) \otimes C)
$$

\medskip
1.  We will study now the  connecting homomorphism  $\partial_3$ from (\ref{les1}) . Let $[\lambda]
\in   H_3(\mathcal{D})$ for some $\lambda \in (\wedge^2 B) \otimes C$.
By the definition of the connecting homomorphism,  see \cite[p. 333]{Rotman},
and the fact that it is well defined we have that $\partial_3([\lambda])$ is the homology class of $d_3(\lambda)$ and $d_3(\lambda) \in B \wedge B$.

Let $\widetilde{d}_*$ be the differential of $\mathcal{D}$.
Then $\widetilde{d}_3(\lambda) = 0$ and for $\lambda = \sum_i b_{1,i} \wedge b_{2,i} \otimes c_i$
we have
$$\widetilde{d}_3(\sum_i b_{1,i} \wedge b_{2,i} \otimes c_i) = \sum_i [b_{1,i}, b_{2,i}] \otimes c_i = 0$$
Thus we can split $\lambda$ as a sum of elements with fixed $c_i$ i.e. we can assume that $c_i = c$ for all $i$ and $\sum_i [b_{1,i}, b_{2,i}] = 0$.
Now $ d_3(\lambda)$ equals
$$ \sum_i ( - [b_{1,i}, b_{2,i}] \wedge c + [b_{1,i},c] \wedge b_{2,i} -  [b_{2,i},c] \wedge b_{1,i})=
\sum_i (  [b_{1,i},c] \wedge b_{2,i} +   b_{1,i}  \wedge [b_{2,i},c] )
$$
where for $b \in B, c \in C$ we identify $b \wedge c$ with $b \otimes c \in B \otimes C$.
Now an element $\underline{a} = \sum_i a_{1,i} \wedge a_{2,i} \in B \wedge B$ belongs to $ker(d_2 |_{B \wedge B})$ precisely when $\sum_i [a_{1,i}, a_{2,i}] = 0$. Thus $$H_2(B, k) \simeq H_2(\mathcal{B}) =  ker(d_2  |_{B \wedge B}) / Im( d_3  |_{B \wedge B \wedge B}) $$ is the set of all homology classes $\underline{a} + Im (d_3)$ and by the above description of $d_3(\lambda)$ for $b_{j,i} = a_{j,i}$ we obtain that $$Im (\partial_3) = H_2(\mathcal{B}) \circ C,$$ where we view $H_2(\mathcal{B})$ as a subquotient of $B \wedge B$ and $\circ$ denotes the diagonal $C$-action on $B \wedge B$ and the $C$-action on $B$ is the adjoint one.

Thus by the long exact sequence we obtain that $H_2(\mathcal{B})_C = H_2(\mathcal{B})/ Im (\partial_3)$ embeds in $H_2(\mathcal{M})$
i.e. there is an embedding $H_0(C, H_2(B,k)) \to H_2(M,k)$.

\medskip
2.  Recall that the complex $\mathcal{D}$  was defined in (\ref{complexes-ses}), i.e.
$$
\mathcal{D} : \ldots  \to \wedge^3 B \otimes C \mapnew{\widetilde{d}_4}  \wedge^2 B \otimes C \mapnew{\widetilde{d}_3} B \otimes C
 \mapnew{\widetilde{d}_2} C \to 0
 $$
  Note that  since $\widetilde{d}_2$ is induced by $d_2$ and $Im (d_2) \subseteq B$ we deduce that $\widetilde{d}_2 = 0$.
  Similarly since $\widetilde{d}_3$ is induced by $d_3$  we get that $\widetilde{d}_3( b_1 \wedge b_2 \otimes c) = [b_1, b_2] \otimes c \in B \otimes C.$
  Hence
  $$
  H_2(\mathcal{D}) = Ker( \widetilde{d}_2) / Im (\widetilde{d}_3) \simeq (B/ [B,B]) \otimes C  \simeq B/ [B,B].
  $$
 The long exact sequence in homology  (\ref{les1}) implies the short exact sequence
 $$
 0 \to H_2(\mathcal{B})/  Im ({\partial}_3)  \to H_2(\mathcal{M}) \to  Ker(\partial_2)  \to 0$$
 We have already identified $H_2(\mathcal{B})/  Im ({\partial}_3)  $ with $H_2(\mathcal{B})_C = H_0(C, H_2(\mathcal{B}))$.

 Now we study the connected homomorphism $\partial_2 :  H_2(\mathcal{D})  \to H_1(\mathcal{B})$ from (\ref{les1}).
 Let $\lambda = b \otimes c \in B \otimes C$ and $[\lambda]$ be the image of $\lambda$ in $B/ [B,B] \otimes C \simeq B/ [B,B]$.
 Then ${\partial}_2([\lambda])$ is the homology class of $d_2(b \otimes c) = [b,c] \in B$
 i.e. is $[b,c] + [B,B]$. Note that $H_1(\mathcal{B}) = B/ [B,B]$. Thus
 $$
  Ker (\partial_2) = H^0(C, B/[B,B]) \simeq H_1(C, B/ [B,B]) = H_1(C, H_1(B,k)).$$

b) Finally suppose that $b_0 \in B \setminus [B,B]$ and $c_0 \in C \setminus \{ 0 \}$ be such that $[b_0, c_0] = 0$.
Then $\lambda = b_0 \otimes c_0 \in B \otimes C$ and $d_2(b_0 \otimes c_0) = [b_0, c_0] = 0$, then
  $\lambda + Im (d_3) \in H_2(\mathcal{M})$ .
  The map
 $H_2(\mathcal{M}) \to H_2(\mathcal{D})$  from (\ref{les1}) sends  $\lambda + Im (d_3)$
 to the homology class $b_0 \otimes c_0 + im (\widetilde{d}_3) = b_0 \otimes c_0 + [B, B] \otimes C \in (B/ [B,B]) \otimes C$.
 The last homology class is not zero since $b_0 \notin [B,B]$ and $c_0 \not= 0$.
\end{proof}

\begin{lemma} \label{derived-fin-gen} The derived subalgebra $[L,L]$ is finitely generated.
\end{lemma}
\begin{proof}
	Indeed, let $F$ be the free Lie algebra with a free basis $v_1, v_2$. Set $v = [v_1, v_2]$.
Then by ~\cite{Ba} $[F,F]$ is a free Lie algebra with a free basis
$X = \{ ad(v_2)^j ad(v_1)^i (v), i \geq 0 , j \geq 0 \}$. Let $a_{i,j}$ be the image of $ad(v_2)^j ad(v_1)^i (v)$ in $L$.  Here $ad(y) (w) = [w,y]$.
Then
\begin{align*}
	a_{0,0} &= v_3, \quad a_{1,0} = [v_3, v_1] =  t_0 v_4, \quad a_{0,1} = [v_3, v_2] =  v_4,\\
	a_{2,0} &= [a_{1,0}, v_1] = [v_1, t_0 v_4] =  v_1(t_0) v_4  +  t_0 [v_1, v_4] =   0,\\
	a_{0,2} &= [a_{0,1}, v_2] =  [v_2, v_4] =  t_1 v_5,\\
    a_{0,3} &= [a_{0,2}, v_2] = [v_2, t_1v_5] = v_2(t_1) v_5 + t_1[ v_2, v_5]= 0,\\
	a_{1,1} &= [a_{1,0}, v_2] = [v_2, t_0 v_4] = v_2(t_0) v_4 + t_0 [v_2, v_4] =   t_0 t_1 v_5,\\
	a_{1,2} &= [a_{1,1}, v_2] = [v_2, t_0 t_1 v_5] = v_2(t_0 t_1) v_5 +  t_0 t_1 [v_2, v_5] =  0.
\end{align*}
The image of $X$ in $L$
yields a finite generating set  $\{v_3,v_4, t_0v_4,t_1v_5,t_0t_1v_5\}.$
\end{proof}

\begin{lemma} \label{fin-pres-lemma}
Let $K = N \leftthreetimes  Q$ be a finitely presented $\mathbb{N}$-graded Lie algebra,
where $N$ and $Q$ are $\mathbb{N}$-graded  subalgebras and $Q=kb$ one dimensional, where $ad(b) \in Der(N)$ is nilpotent.
Then $N$ is finitely presented.
	\end{lemma}

\begin{proof} We note first that $H_1(K,k) = K/ [K,K]$ is finite dimensional,
hence its subspace $H_0(Q, H_1(N,k)) = N/ ( [N,N] + [N,Q]) = N/ [N,K]$ is finite dimensional.
Thus $H_1(N,k)$ is finitely generated as $U(Q)$-module.
	
Consider the  short exact sequence  given by Lemma \ref{split}
$$ 0 \to H_0(Q, H_2(N,k)) \to H_2(K,k) \to H_1(Q, H_1(N, k)) \to 0$$
Since $K$ is finitely presented $H_2(K,k)$ is finite dimensional, hence $ H_0(Q, H_2(N,k)) $ is finite dimensional.
Since all modules (homologies) here are $\mathbb{N}$-graded ones  from $H_0(Q, H_2(N,k))$ finite dimensional
we deduce that $H_2(N,k)$ is finitely generated as $U(Q)$-module.
Since $Q = k b$ and $ad(b)^s = 0$ for some $s \geq 1$ (i.e.  $ad(b)$ is nilpotent),
we have that the action of $U(Q)$ on $H_2(N,k)$ factors through $k[b]/ ( b^s)$, hence $H_2(N,k)$ is finite dimensional.
Since $N$ is a $\mathbb{N}$-graded Lie algebra this means that $N$ is finitely presented.
\end{proof}

\noindent
{\bf Definition.}
{ \it We apply the shift $\tau^{i-1}$, $i\ge 1$, for all components~\eqref{decomp0} and get
\begin{equation}\label{decomp_i}
L_{i-1} = (A_i \leftthreetimes L_i) \leftthreetimes Q_i= M_i \leftthreetimes Q_i,\quad\text{where}\quad M_i: = A_i \leftthreetimes L_i
\end{equation}
and $L_0:=L$, $A_1:=A$. Thus,
$L_{i-1}:=\tau^{i-1} (L)=\tau^{i-1} (\langle v_1,v_2\rangle)= \langle v_{i}, v_{i+1}\rangle $,
$Q_i:=\tau^{i-1}(kv_1)= kv_i$, and
$A_i:=\tau^{i-1}A= \{ t_{i-1} t_i^{\alpha_i} \ldots t_j^{\alpha_j} v_{j+4} \mid  \alpha_i, \ldots, \alpha_j \in \{ 0, 1 \}  \}$.
}

\begin{cor} \label{cor-fin-pres}

a)  If $L$ is finitely presented then $M_1$, $L_1$ and $M_2$ are finitely presented.

b) If $\widetilde{L}$ is finitely presented then $L$ is finitely presented.
\end{cor}

\begin{proof} a) Using~\eqref{decomp_i} in case $i=1$,
$L = M_1   \leftthreetimes Q_1$ where $Q_1 = k v_1$ and $ad(v_1) \in Der (M_1)$ nilpotent.
Recall that $L$ is $\mathbb{N}$-graded.
Applying Lemma \ref{fin-pres-lemma},  $M_1$ is finitely presented.
Then $L_1=\tau(L)$ and $M_2=\tau(M_1)$ are also finitely presented.
	
b) By item 3) of Theorem~\ref{just_infinite}, we get $\widetilde L=k\partial_1\rightthreetimes L_1$ and   $L_1\cong L$.
These algebras are $\mathbb N$-graded (Theorem~\ref{Ngrading}) and $\partial_1^2=0$.
It remains to apply Lemma \ref{fin-pres-lemma}.
\end{proof}

\begin{lemma} \label{abelianization}

a) The vector space $A_1/ [A_1, L_1]$ is spanned by  the  image of $[v_1, v_3]=t_0v_4$.

b) The vector space $M_1/[M_1, M_1]$ is spanned by images of $\{v_2, v_3, [v_1, v_3] = t_0 v_4\}$.
\end{lemma}
\begin{proof}
Recall that $L = L_{0} = (A_1 \leftthreetimes L_1) \leftthreetimes Q_1$ and $A_1$ is an abelian ideal in $L$.
Then for $M_1 = A_1 \leftthreetimes L_1$ we have that
\begin{equation}\label{sum}
 M_1/ [M_1,M_1] \simeq (A_1/ [A_1, L_1]) \oplus (L_1/ [L_1, L_1]).
\end{equation}

Since  $M_1$ is the ideal of $L$ generated by $v_2$, $M_1/ [M_1,M_1]$
is spanned by the images of $\{ ad(v_1)^j (v_2) \}_{j \geq 0}$~\cite{Ba}.
Then we get
\begin{align*}
&ad(v_1) (v_2) =[v_2, v_1] =  v_3,\quad ad(v_1)^2 (v_2) = ad(v_1) (v_3) = [v_3, v_1] =  t_0 v_4,\\
&ad(v_1)^3(v_2) = [t_0 v_4, v_1] = [v_1, t_0 v_4] = v_1(t_0) v_4 + t_0[v_1, v_4] = 0.
\end{align*}
Hence,  $ M_1/ [M_1,M_1]$ is spanned by images of $\{v_2,v_3,t_0v_4\}$.

Since $L/ [L,L]$ has a $k$-basis $\{v_1, v_2\}$, applying the shift monomorphism $\tau:L\cong L_1$,
a $k$-basis of $L_1/ [L_1, L_1]$ is $\{v_2, v_3\}$, yielding  a basis of the second  summand of~\eqref{sum}.

By our computations, $t_0v_4$ has degree 3 in the generators $\{v_1,v_2\}$.
There is one more element of degree 3, namely $[[v_1,v_2],v_2]=v_4\notin A_1$.
So,  the  degrees of  the  monomials $X_A$ of $A_1=A$ are at least 3 and  the   degrees of monomials for $[A_1,L_1]$ are at least 4.
Using Theorem~\ref{Ngrading}, the image of $t_0v_4$ is a basis of the first component~\eqref{sum}.

\end{proof}

\begin{lemma} \label{additional-lemma} Let $0 \to W_0 \to W \mapnew{p} k \to 0$ be a short exact sequence
of right  $U(Q)$-modules, where  $ Q$ is one dimensional Lie algebra.
Assume further that there is an element $w \in H^0(Q, W) = \{ v \in W \ | \ v  Q = 0 \} $ such that $p(w) \not= 0$.
Then there is a short exact sequence of $k$-vector spaces of  coinvariants
$$0 \to (W_0)_{Q}  \to (W)_{Q}   \to (k)_{Q} = k \to 0.$$
\end{lemma}

\begin{proof} Observe that $Q$ acts trivially
on $k \simeq W/ W_0$. Indeed $W/ W_0$ is a vector space spanned by $p(w)$ and since $Q$ acts trivially on $w$, $Q$ acts trivially on $p(w)$.

Recall that for any right $U(Q)$-module $D$ we have that
$D_Q = D \otimes_{U(Q)} k \simeq Tor_0^{U(Q)}(D, k).$
Then the long exact sequence in homology gives an exact complex
$$0= H_2(Q, k) \to H_1(Q, W_0) \to H_1(Q, W)  \to $$ $$H_1(Q, k) \mapnew{\partial} (W_0)_{Q}  \to (W)_{Q}   \to (k)_{Q} = k \to 0$$
Since $dim_k Q = 1$, $H_1(Q, k) \simeq Q/[Q,Q] = Q$ is 1-dimensional over $k$. There are two cases :

1. If $\partial = 0$ then $0 = Im (\partial) = Ker ((W_0)_{Q}  \to (W)_{Q})$, so the lemma is proved.

2. If $\partial \not= 0$ then $0 = Ker(\partial) = Im (H_1(Q, W)  \to H_1(Q, k) )$, hence $ H_1(Q, W_0) \to H_1(Q, W) $ is an isomorphism.
By Lemma \ref{duality} we conclude that the inclusion $W_0 \to W$ induces an isomorphism $H^0(Q, W_0) \to H^0(Q, W)$.
Identifying $W_0$ with its image in $W$ we conclude that  $H^0(Q, W_0) = H^0(Q, W)$, a contradiction with
$w \in H^0(Q, W) \setminus H^0(Q, W_0)$.
\end{proof}

\medskip
{\bf Proof of the Main Theorem } We split the proof in several steps.  Recall that
$
L_{i-1} = (A_i \leftthreetimes L_i) \leftthreetimes Q_i
$
was defined before Corollary  \ref{cor-fin-pres} and $L = L_0$, $M_i = A_i \leftthreetimes L_i$.

\medskip
Step 1.	By Lemma \ref{split} there is
	a short exact sequence
\begin{equation} \label{ses-kalt}  0 \to H_2(M_2,k)_{Q_2} \to H_2(L_1,k) \to H_1(Q_2, H_1(M_2,k)) \to 0 \end{equation}
 We will identify first $H_1(Q_2, H_1(M_2,k)) $, will show it is one dimensional and will find one element in $H_2(L_1,k)$ that maps
to a non-zero element of $H_1(Q_2, H_1(M_2,k))$.

Using the shift isomorphism $\tau: M_1 \to M_2$ and Lemma~\ref{abelianization}, we get
\begin{equation}\label{H1short}
H_1(M_2,k) \simeq M_2/ [M_2, M_2]  = k v_3 \oplus k v_4 \oplus k w_2,
\end{equation}
where $w_2 := [v_2, v_4] =  t_1 v_5$.  Since $Q_2=kv_2$, we can use Lemma \ref{duality} and get
$$H_1(Q_2, H_1(M_2,k)) \simeq H^0(Q_2, M_2/ [M_2, M_2]).$$

We proceed by calculating $H^0(Q_2, M_2/ [M_2, M_2])$, i.e. the subspace of elements of $M_2/ [M_2, M_2]$ on which the adjoint action of $Q_2$ is trivial.
Consider the  adjoint action of $Q_2=kv_2$ on the basis of $M_2 / [M_2, M_2] $ given by~\eqref{H1short}. 
It sends $v_3$ to $[v_3, v_2] =   v_4$, $v_4$ to $[v_4, v_2]  =  w_2 = t_1 v_5$ and 
$w_2$ to $[w_2, v_2] = [t_1 v_5, v_2] =  v_2(t_1) v_5  + t_1 [v_2, v_5] = 0 + 0 = 0 $.
Hence $$ H^0(Q_2, M_2/ [M_2, M_2]) = k w_2.$$

Since $[w_2, v_2] = 0$,  by  Lemma \ref{split}, b) applied for $B = M_2$, $C = Q_2$, $M = L_1$, $b_0 = w_2$, $c_0 = v_2$, we have
$w_2 \otimes v_2 \in M_2 \otimes Q_2$ is a cycle,
whose image in $H_2(M_2 \leftthreetimes Q_2,k)$
maps to a non-zero element of $H_1(Q_2, H_1(M_2,k))$ i.e. in $L_1 = M_2 \leftthreetimes Q_2$ we get
$$r := w_2\wedge  v_2\in L_1\wedge L_1, \qquad [r] \in H_2(L_1,k),$$
and the image of the latter (see~\eqref{ses-kalt})  in $H_1(Q_2, H_1(M_2,k))$ is non-zero.

\medskip
Step 2.
Recall that $$L = (A_1 \leftthreetimes L_1) \leftthreetimes Q_1.$$ Consider the adjoint action of $Q_1$ on $L_1$ modulo $A_1 =  A$.
This  induces action  of $Q_1$ on $H_2(L_1,k)$.
Recall that $$L_1 = M_2 \leftthreetimes Q_2$$ and note that  $M_2$ is $Q_1$-invariant  (under the action modulo $A_1$).

By (\ref{ses-kalt}) we consider $H_2(M_2,k)_{Q_2}$ as a $k$-vector subspace of $H_2(L_1,k)$.
In fact $H_2(M_2,k)_{Q_2}$ is the image of the map $\rho : H_2(M_2, k) \to H_2(L_1, k)$ induced by the embedding of $M_2$ in $L_1$.

By viewing $H_2(M_2,k)$ as a  subquotient of $M_2 \wedge M_2$ and  $H_2(L_1,k)$
as a subquotient of  $ L_1 \wedge L_1$  the map $\rho$ is induced by the embedding $M_2 \wedge M_2 \to L_1 \wedge L_1$,
but itself is not necessary embedding.
 Furthermore $M_2 \wedge M_2$ is a $Q_1$-submodule of $L_1 \wedge L_1$ via the diagonal $Q_1$-action.
Thus $H_2(M_2,k)_{Q_2}$  is a $Q_1$-invariant  $k$-vector subspace of $H_2(L_1,k)$.

Recall that  $[r]$ is the image of $r$ in $H_2(L_1,k)$ and assume that $Q_1$ acts trivially on $[r]$ i.e. sends it to zero.
Then by (\ref{ses-kalt}) and Lemma \ref{additional-lemma}
we get a short exact sequence
\begin{equation} \label{sofia} 0 \to
(H_2(M_2,k)_{Q_2} )_{Q_1} \to H_2(L_1,k)_{Q_1} \to H_1(Q_2, H_1(M_2,k))_{Q_1} \to 0\end{equation}
and since $Q_1$ acts trivially on $[r]$ we deduce that $$ H_1(Q_2, H_1(M_2,k))_{Q_1} \simeq  H_1(Q_2, H_1(M_2,k))$$ has dimension one.

We show now that $Q_1$ acts trivially on $[r]$.
Note that the relation $r$ in the homology group $H_2(L_1,k)$ is represented by the cycle $ w_2 \otimes v_2 \in M_2 \otimes Q_2 \subset L_1 \wedge L_1$.
And the $v_1$ action is the diagonal one in $L_1 \wedge L_1$ induced by the adjoin action modulo $A$
i.e. $[w_2, v_1] $ mod $A$ is $[t_1 v_5, v_1]  =  v_5 + t_1[v_5, v_1]  = v_5 + t_1 t_0 t_1 t_2 v_6 = v_5$
and $[v_2, v_1] =  v_3$, so $w_2 \wedge v_2$ is sent by the  (diagonal) action of $v_1$ to $ v_5 \wedge v_2 + w_2 \wedge v_3$.
We have to show  this is a boundary i.e. it is an image of $\wedge^3 L_1$
 under the differential~\eqref{diff3}
and this would imply that the action of $Q_1$ on $[r]$ is trivial. Note that
$$d_3(v_4 \wedge v_2 \wedge v_3) = [v_4, v_2] \wedge v_3 + [v_4, v_3] \wedge v_2 + [v_2, v_3] \wedge v_4 =$$
$$= w_2 \wedge v_3  + v_5 \wedge v_2 + v_4 \wedge v_4 = w_2 \wedge v_3  + v_5 \wedge v_2. $$

In the rest of step 2
we give an alternative proof of the fact that the action of $Q_1$ on $[r]$ is trivial  as follows:
consider $F$ the free Lie algebra on the letters $v_2$ and $v_3$ and think of $r = [v_3, v_2, v_2, v_2]$, a left normed commutator,
as an element of $F$. Let $$\varphi : F \to L_1$$ be the Lie algebra homomorphism that is the identity on $v_2$ and $v_3$
and let $R = ker(\varphi)$. Let  $\theta \in Der(L_1)$ be induced by $ad(v_1)$ (recall this is modulo $A$).
Note that in $L$ we have $[v_3, v_1] \in A$ and $[v_2, v_1] = v_3 \in L_1$, so $\theta(v_3) = 0$ and $\theta(v_2) =  v_3$.
Write $\widetilde{\theta} \in Der (F)$ such that $\widetilde{\theta}(v_3) = 0$ and $\widetilde{\theta}(v_2) = v_3$.
Then $$\widetilde{\theta}(r) = [\widetilde{\theta}(v_3), v_2, v_2, v_2] + [v_3, \widetilde{\theta}(v_2), v_2, v_2] + [v_3, v_2, \widetilde{\theta}(v_2), v_2] + [v_3, v_2, v_2, \widetilde{\theta}(v_2)] =$$ $$   0 + 0 +  [v_3, v_2, v_3, v_2] + [v_3, v_2, v_2, v_3] =  [[v_3, v_2], [v_3, v_2]] = 0.$$
Now by the Hopf formula $$H_2(L_1, k) \simeq ([F,F] \cap R)/ [R,F]$$ and the image of $r$ in the quotient $([F,F] \cap R)/ [R,F]$
is precisely $[r] \in H_2(L_1,k)$.

\medskip
Step 3.  By Lemma \ref{newMay}  applied for $M_2 = A_2  \leftthreetimes L_2$ there is an epimorphism
$$H_2(M_2,k) \mapnew{\pi_*} H_2(L_2,k)$$
induced by the epimorphism $M_2 \to L_2$
and a monomorphism
$$H_2(L_2,k) \mapnew{i_*} H_2(M_2,k)$$
induced by the inclusion map $L_2 \to M_2$ such that $\pi_* \circ i_* = id$.

Recall that $$L_1 = ( A_2  \leftthreetimes L_2)  \leftthreetimes Q_2 = M_2  \leftthreetimes Q_2$$ and $Q_2$ acts on $M_2$
via the adjoint action sending $A_2$ to $A_2$ (recall $A_2$ is an ideal in $L_1$) and sending  $L_2$ to $M_2 = A_2  \leftthreetimes  L_2$.
Hence the diagonal $Q_2$-action sends $\wedge^2 A_2$ to itself and sends $L_2 \otimes A_2$ to $(\wedge^2 A_2) \oplus (L_2 \otimes A_2) \subset \wedge^2 M_2$.
In particular considering $Ker(\pi_*)$ as a subquotient of $(\wedge^2 A_2) \oplus (L_2 \otimes A_2)$
we deduce that $Ker(\pi_*)$ is $Q_2$-invariant
 $k$-subspace of $H_2(M_2,k)$.  Then we have a short exact sequence of $Q_2$-modules
$$0 \to Ker (\pi_*) \to H_2(M_2,k) \mapnew{\pi_*} H_2(L_2,k) \to 0.$$
 Note that this induces an action of $Q_2$ on $H_2(L_2,k)$ that coincides with the action induced by the adjoint action of $Q_2$ on $L_2$
 modulo $A_2$ and the above exact sequence splits.
\noindent
This induces an exact sequence
 because the functor of coinvariants is right exact
\begin{equation} \label{ses-kalt2}
( Ker(\pi_*) )_{Q_2} \to H_2(M_2,k)_{Q_2} \to H_2(L_2,k)_{Q_2} \to 0,\end{equation}
where the first map is not necessary injective.

Let $$V = H_2(M_2,k)_{Q_2}.$$ As discussed in Step 2 $V$ is $Q_1$-invariant but it is important to stress that  the sequence (\ref{ses-kalt2})
 is not $Q_1$-invariant. Indeed the first module from the left in (\ref{ses-kalt2})
 is a subquotient of $(\wedge^2 A_2) \oplus (L_2 \otimes A_2)$ and the last module in (\ref{ses-kalt2}) is a subquotient of $\wedge^2 L_2$.
 The $Q_1$-action on $V$ is induced by the $Q_1$-action on $L_1 = M_2 \leftthreetimes Q_2$
 (that is the adjoint action reduced  modulo $A = A_1$),
 where $M_2$ is $Q_1$-invariant  and thus the $Q_1$-action sends $A_2$ to $L_2$ and sends $L_2$ to $0$. 
  Indeed since $\{ t_1t_2^{z_2} \ldots t_j^{ z_j} v_{j+4} \ | \ j \geq 1 \}$ is a basis of $A_2$ as $k$-vector space and 
 $[v_1, t_1t_2^{z_2} \ldots t_j^{ z_j} v_{j+4}] = v_1(t_1t_2^{z_2} \ldots t_j^{ z_j}) v_{j+4} \equiv t_2^{z_2} \ldots t_j^{ z_j} v_{j+4}$ 
 modulo $A$ we have that the $Q_1$-action sends $A_2$ to $L_2$. Since $L_2 = \langle v_3, v_4 \rangle$ and $[v_1, v_3] = t_0 v_4 \equiv 0 $ 
 modulo $A$ and $[v_1, v_4] = t_0 t_1 v_5 \equiv 0 $ modulo $A$ we conclude that the $Q_1$-action sends  $L_2$ to $0$.
 Thus (via the diagonal $Q_1$-action) $A_2 \wedge A_2$ goes to
  $L_2 \wedge A_2 \subset \wedge^2 M_2$,  $L_2 \otimes A_2$ goes to $\wedge^2 L_2$ and $\wedge^2 L_2$ goes to 0 i.e. (\ref{ses-kalt2})
 is not $Q_1$-invariant but \begin{equation} \label{newMay2} Q_1 \hbox{ acts trivially on }L_2 \wedge L_2, \hbox{ hence on }H_2(L_2,k)_{Q_2},\end{equation}
where we think of $H_2(L_2,k)_{Q_2}$ as a subquotient of $L_2 \wedge L_2$.

 Let  $W$ be a lifting of $H_2(L_2,k)_{Q_2}$ in $V = H_2(M_2,k)_{Q_2}$ that is a subquotient of $\wedge^2 L_2$.
 Since the $Q_1$-action ( mod $A_1$) on $\wedge^2 L_2$ is trivial, $Q_1$ acts trivially on $W$
 i.e. $$W \subseteq H^0(Q_1, V) = \{ c \in V ~ | ~ Q_1 \hbox{ acts trivially on } c \}.$$

 Assume that $L$ is {\it finitely presented}.
 Then by Corollary \ref{cor-fin-pres} $M_1$, $L_1$ and $M_2$ are finitely presented and $H_2(M_2, k)$ is finite dimensional.
 In particular $V$ is finite dimensional.

Let $\mu : V \to V$ be the map that sends an element to its image under the action (induced by the adjoint action mod $A_1$)
of the generator $v_1$ of $Q_1$. Then $$H_0(Q_1, V) = Coker ( \mu)\hbox{ and }H^0(Q_1, V) = Ker (\mu).$$
Since $V$ is finite dimensional we have $dim_k Ker(\mu) = dim_k Coker( \mu)$, hence
$$
dim_k \ (H_2(M_2,k)_{Q_2} )_{Q_1} = dim_k \  V _{Q_1} =  dim_k  \ Coker( \mu) = dim_k \  Ker(\mu) $$ $$
=dim_k \ H^0(Q_1, V) \geq dim_k W = dim_k \ H_2(L_2,k)_{Q_2}.
$$
Then by (\ref{sofia}) and since by Step 1 $dim_k H_1(Q_2, H_1(M_2,k)) = 1$  we obtain
\begin{align*} \label{important*}
dim_k \ H_2(L_1, k)_{Q_1} &= 1+  dim_k \ (H_2(M_2,k)_{Q_2} )_{Q_1} \geq 1 + dim_k \ H_2(L_{2},k)_{Q_{2}}\\
&=1 + dim_k \ H_2(L_{1},k)_{Q_{1}},
\end{align*}
the last equality being valid by application of the shift isomorphism $\tau: L_1\to L_2$.
Hence $dim_k \ H_2(L_{1},k)_{Q_{1}}=\infty$ and we get
\begin{equation*} 
dim_k \ H_2(L, k)=dim_k \ H_2(L_1, k)  \geq  dim_k \ H_2(L_1, k)_{Q_1} =  \infty.
\end{equation*}
But the assumption that $L$ is finitely presented implies that $dim_k \ H_2(L, k) < \infty$.
The contradiction proves that $L$ is not finitely presented.

\medskip
Step 4. By Corollary \ref{cor-fin-pres}, part b) $\widetilde{L}$ is not finitely presented as well.


\end{document}